\documentclass{imsart}

\usepackage{amsthm,amsmath}

\startlocaldefs

\usepackage{subfigure}
\usepackage{tikz}
\newcommand{\nicefrac}[2]{#1/#2}
\usepackage{csquotes}
\usepackage{amssymb}
\usepackage{soul}

\newtheorem{lemma}{Lemma}
\newtheorem{proposition}{Proposition}
\newtheorem{example}{Example}
\newtheorem{theorem}{Theorem}
\newtheorem{corollary}{Corollary}
\newtheorem{conjecture}{Conjecture}
\newcommand{\xx}{1}
\newcommand{\yy}{1}
\newcommand{\stage}[2]{\tikz{\node[shape=circle,draw,inner sep=1pt,fill=#1]{$v_{#2}$};}} 
\newcommand{\leaf}{\tikz{\node[shape=circle,draw,inner sep=3pt,fill=white] {};}}
\newcommand\independent{\protect\mathpalette{\protect\independenT}{\perp}}
\def\independenT#1#2{\mathrel{\rlap{$#1#2$}\mkern2mu{#1#2}}}
\DeclareMathOperator{\Multi}{Multi}

\usepackage{times}
\usepackage{bm}
\usepackage{natbib}
\usepackage[capitalize]{cleveref}
\crefformat{equation}{(#2#1#3)}
\crefrangeformat{equation}{(#3#1#4) to~(#5#2#6)}
\crefmultiformat{equation}{(#2#1#3)}%
{ and~(#2#1#3)}{, (#2#1#3)}{ and~(#2#1#3)}
\crefmultiformat{fig}{Figures #2#1#3}{ and~#2#1#3}{, #2#1#3}{ and~#2#1#3}
\usepackage{autonum}

\usepackage[plain,noend]{algorithm2e}
\newcommand{\ptheta}{p_{\theta}}
\newcommand{\R}{\mathbb{R}}
\newcommand{\T}{\mathcal{T}}

\endlocaldefs

\begin{document}

\begin{frontmatter}

\title{The curved exponential family of a staged tree}
\runtitle{The curved exponential family of a staged tree}


\author{\fnms{Christiane} \snm{G\"{o}rgen}\ead[label=e1]{goergen@math.uni-leipzig.de}}
\address{Mathematisches Institut, Universit\"at Leipzig, Germany \\\printead{e1}}
\author{\fnms{Manuele} \snm{Leonelli}\corref{}\ead[label=e2]{manuele.leonelli@ie.edu}}
\address{School of Human Sciences and Technology, IE University, Madrid, Spain \\\printead{e2}}

\author{\fnms{Orlando} \snm{Marigliano}\ead[label=e3]{orlandom@kth.se}}
\address{KTH Royal Institute of Technology, Stockholm, Sweden \\\printead{e3}}

\runauthor{G\"{o}rgen, Leonelli and Marigliano}

\begin{abstract}
Staged tree models are a discrete generalization of Bayesian networks. We show that these form curved exponential families and derive their natural parameters, sufficient statistic, and cumulant-generating function as functions of their graphical representation. We give necessary and sufficient graphical criteria for classifying regular subfamilies and discuss implications for model selection.
\end{abstract}

\begin{keyword}[class=MSC]
\kwd[Primary ]{62H99}
\kwd[; secondary ]{68T99}
\end{keyword}

\begin{keyword}
\kwd{Bayesian Information Criterion}
\kwd{Chain Event Graph}
\kwd{Curved Exponential Family}
\kwd{Graphical Model}
\kwd{Staged Tree}
\end{keyword}



\end{frontmatter}

\section{Introduction}

Staged trees define statistical models which can account for a variety of partial and asymmetric conditional independence statements between discrete events \citep{Collazo.etal.2018}.  The use of these graphical models in applications is constantly increasing \citep{Barclay.etal.2013a,Collazo.Smith.2015,Keeble2017} and free software for practitioners is newly available \citep{rpackage}. However, their formal properties have only recently been studied for the first time \citep{Goergen.Smith.2017}. We now extend this formal study by 1.\,proving that in general staged tree models form curved exponential families and by 2.\,expressing their natural parameters, sufficient statistic, and cumulant-generating function as a function of the underlying graphical representation. We also give a graphical criterion under which they form regular exponential families.

Because exponential families exhibit a multitude of desirable inferential properties \citep{Kass.Vos.1997}, similar characterisations for other graphical models have been studied.
For instance, undirected graphical models with no hidden variables are regular exponential families \citep{Lauritzen.1996}, Bayesian networks \citep[e.g.][]{Koller.Friedman.2009} are curved exponential families, and directed graphical models with hidden variables are stratified exponential families \citep{Geiger.etal.2001}. 
These results are critical for model selection techniques. In particular, \citet{Haughton.1988} proved that for curved exponential families the Bayesian information criterion \citep{Schwarz.1978} is an asymptotically valid rule.
To this date model selection for staged trees has usually been carried out in a Bayesian fashion by selecting the maximum a posteriori model \citep{Freeman.Smith.2011a,Barclay.etal.2013a,Cowell.Smith.2014}. Now, our results make it theoretically sound to use the Bayesian information criterion as well, and its implementation in the \texttt{R} package \texttt{stagedtrees} \citep{rpackage} is thus justified. This criterion has already been employed by \citet{Silander.etal.2013} although its asymptotic geometric validity was not assured at the time.

By expressing every staged tree model explicitly as a curved exponential family, we in particular achieve such an explicit description for every discrete Bayesian network. To the best of our knowledge, this is still missing in the literature.
We demonstrate on small-scale examples how a given Bayesian network's exponential form can be stated in staged-tree language.
We can then also specify new graphical criteria under which the model is a regular exponential family.


\section{Staged Trees}
\label{sec:ST}

\subsection{Probability trees as graphical statistical models}
\label{sub:probtrees}

Probability trees are highly intuitive depictions of unfoldings of discrete events \citep{Shafer.1996} and have been used in a variety of real-world applications \citep{Smith.2010,Collazo.etal.2018}.

Let $\T$ be a directed, rooted tree graph where every vertex has either no or at least two emanating edges. For simplicity, we number the inner (non-leaf) vertices $v_1,\ldots,v_k$, with $v_1$ denoting the root, and count the edges emanating from each vertex with indices $j=1,\ldots,\kappa_i$ for all $i=1,\ldots,k$.
To every edge we assign a positive probability label ${\theta_{ij}\in(0,1)}$ such that the total sum of all labels  belonging to the same vertex is always equal to one, $\sum_{j=1}^{\kappa_i}\theta_{ij}=1$. Such a labelled tree graph is called a \emph{probability tree}.

For all root-to-leaf paths $x$ of $\T$ we define $\ptheta(x)=\prod_{i=1}^k\prod_{j=1}^{\kappa_i}\theta_{ij}^{\alpha_{ij}(x)}$ where $\alpha_{ij}(x)$ is one if $x$ passes through the $j$th edge emanating from $v_i$ and zero otherwise. By the constraints on the $\theta_{ij}$, the function $p_\theta$ is a probability distribution on the set of all root-to-leaf paths. We denote the size of this set by $n$.
A \emph{probability tree model} is then the set of all such probability distributions $\ptheta$ on $n$ atoms, for varying labels $\theta$. The set of parameters $\theta$ satisfying the constraints above is the \emph{parameter space} $\Theta_{\T}$ of the model. It equals the product $\times_{i=1}^k\Delta_{\kappa_i-1}^\circ$ of the $k$ open probability simplices $\Delta_{\kappa_i-1}^\circ=\{(t_1,\ldots,t_{\kappa_i})\in\R^{\kappa_i}\mid\sum_{r=1}^{\kappa_i} t_r=1\text{ and }0<t_r<1\text{ for all }r=1,\ldots,\kappa_i\}$.
Thus, the parameter space $\Theta_{\T}$ has dimension $d=\sum_{i=1}^k(\kappa_i-1)$.

\subsection{Probability trees and conditional independence}
\label{sub:stagedtrees}

An event in a probability tree is a set of root-to-leaf paths, often specified by shared vertices or edges. Conditional independence relationships between events can be visualized using a simple type of colouring of these vertices in the following way.

A \emph{staged tree} is a probability tree together with an equivalence relation on the vertex set such that two vertices are in the same stage if and only if their outgoing edges have the same attached probabilities. \Cref{fig:staged&CEG} decpits a first example. A \emph{staged tree model} is a submodel of a probability tree model where probability simplices in the parameter space have been identified with each other according to the equivalence relation. 

\begin{figure}
\centering
\subfigure[][A staged tree for three binary random variables representing the collider Bayesian network $X\to Z\leftarrow Y$.\label{fig:colliderstaged}]{%
\scalebox{0.8}{
\begin{tikzpicture}
\renewcommand{\xx}{2}
\renewcommand{\yy}{0.7}
\node (v1) at (0*\xx,0*\yy) {\stage{white}{1}};
\node (v2) at (1*\xx,1.5*\yy) {\stage{cyan}{2}};
\node (v3) at (1*\xx,-1.5*\yy) {\stage{cyan}{3}};
\node (v4) at (2*\xx,3*\yy) {\stage{white}{4}};
\node (v5) at (2*\xx,1*\yy) {\stage{white}{5}};
\node (v6) at (2*\xx,-1*\yy) {\stage{white}{6}};
\node (v7) at (2*\xx,-3*\yy) {\stage{white}{7}};
\node (l1) at (3*\xx,3.5*\yy) {\leaf};
\node (l2) at (3*\xx,2.5*\yy) {\leaf};
\node (l3) at (3*\xx,1.5*\yy) {\leaf};
\node (l4) at (3*\xx,0.5*\yy) {\leaf};
\node (l5) at (3*\xx,-0.5*\yy) {\leaf};
\node (l6) at (3*\xx,-1.5*\yy) {\leaf};
\node (l7) at (3*\xx,-2.5*\yy) {\leaf};
\node (l8) at (3*\xx,-3.5*\yy) {\leaf};
\draw[->] (v1) -- node [above, sloped] {$X=0$} (v2);
\draw[->] (v1) -- node [below, sloped] {$X=1$} (v3);
\draw[->] (v2) --  node [above, sloped] {$Y=0$} (v4);
\draw[->] (v2) -- node [below, sloped] {$Y=1$} (v5);
\draw[->] (v3) -- node [above, sloped] {$Y=0$} (v6);
\draw[->] (v3) -- node [below, sloped] {$Y=1$} (v7);
\draw[->] (v4) -- node [above, sloped] {$Z=0$} (l1);
\draw[->] (v4) -- node [below, sloped] {$Z=1$} (l2);
\draw[->] (v5) -- node [above, sloped] {$Z=0$} (l3);
\draw[->] (v5) -- node [below, sloped] {$Z=1$} (l4);
\draw[->] (v6) -- node [above, sloped] {$Z=0$} (l5);
\draw[->] (v6) -- node [below, sloped] {$Z=1$} (l6);
\draw[->] (v7) -- node [above, sloped] {$Z=0$} (l7);
\draw[->] (v7) -- node [below, sloped] {$Z=1$} (l8);
\end{tikzpicture}}%
}
\qquad
\subfigure[][A staged tree representing the conditional independence relation $X\independent Y$ and ${Z\independent X\mid Y}$, and the implied relation $X\independent Z$.
\label{fig:staged}]{
\scalebox{0.8}{
\begin{tikzpicture}
\renewcommand{\xx}{2}
\renewcommand{\yy}{0.7}
\node (v1) at (0*\xx,0*\yy) {\stage{white}{1}};
\node (v2) at (1*\xx,1.5*\yy) {\stage{cyan}{2}};
\node (v3) at (1*\xx,-1.5*\yy) {\stage{cyan}{3}};
\node (v4) at (2*\xx,3*\yy) {\stage{red!70}{4}};
\node (v5) at (2*\xx,1*\yy) {\stage{green!50}{5}};
\node (v6) at (2*\xx,-1*\yy) {\stage{red!70}{6}};
\node (v7) at (2*\xx,-3*\yy) {\stage{green!50}{7}};
\node (l1) at (3*\xx,3.5*\yy) {\leaf};
\node (l2) at (3*\xx,2.5*\yy) {\leaf};
\node (l3) at (3*\xx,1.5*\yy) {\leaf};
\node (l4) at (3*\xx,0.5*\yy) {\leaf};
\node (l5) at (3*\xx,-0.5*\yy) {\leaf};
\node (l6) at (3*\xx,-1.5*\yy) {\leaf};
\node (l7) at (3*\xx,-2.5*\yy) {\leaf};
\node (l8) at (3*\xx,-3.5*\yy) {\leaf};
\draw[->] (v1) -- node [above, sloped] {$X=0$} (v2);
\draw[->] (v1) -- node [below, sloped] {$X=1$} (v3);
\draw[->] (v2) --  node [above, sloped] {$Y=0$} (v4);
\draw[->] (v2) -- node [below, sloped] {$Y=1$} (v5);
\draw[->] (v3) -- node [above, sloped] {$Y=0$} (v6);
\draw[->] (v3) -- node [below, sloped] {$Y=1$} (v7);
\draw[->] (v4) -- node [above, sloped] {$Z=0$} (l1);
\draw[->] (v4) -- node [below, sloped] {$Z=1$} (l2);
\draw[->] (v5) -- node [above, sloped] {$Z=0$} (l3);
\draw[->] (v5) -- node [below, sloped] {$Z=1$} (l4);
\draw[->] (v6) -- node [above, sloped] {$Z=0$} (l5);
\draw[->] (v6) -- node [below, sloped] {$Z=1$} (l6);
\draw[->] (v7) -- node [above, sloped] {$Z=0$} (l7);
\draw[->] (v7) -- node [below, sloped] {$Z=1$} (l8);
\end{tikzpicture}}%
}
\caption{Two discrete graphical models. Vertices which are in the same stage have been assigned the same colour. In this picture, white-coloured vertices are not in the same stage with any other vertex.}\label{fig:staged&CEG}
\end{figure}
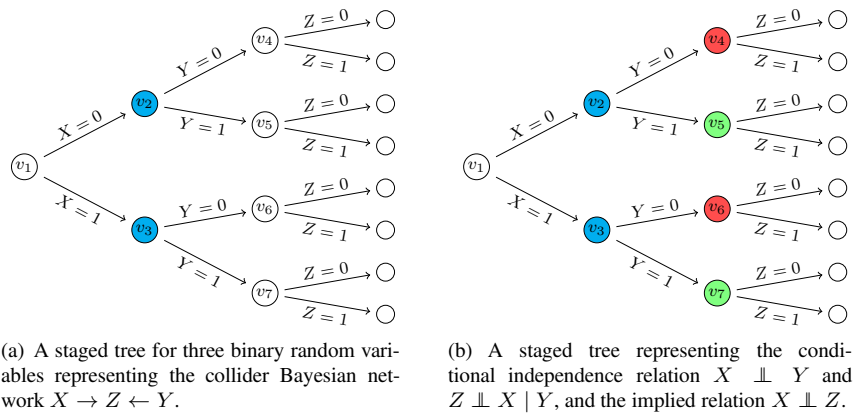

For example, if a tree depicts the product state space of a vector of discrete random variables $Y_i$, then every inner vertex represents a
random variable ${Y_i\mid Y_{[i-1]}=y_{[i-1]}}$
conditional on specific values taken by its ancestors, every edge corresponds to a state $y_i$ of that variable, and every edge can be labelled by the
conditional probability ${P(Y_i = y_i \mid Y_{[i-1]}=y_{[i-1]})}$
of the random variable being in that state given the particular ancestor configuration. Here, $[i-1]$ is the set of indices $\{1,\dotsc, i-1\}$.
The stage relations on the tree can then be used to model conditional independence relations such as ${P(Y_i=y_i\mid y_{[i-1]}) = P(Y_i=y_i\mid y'_{[i-1]})}.$

To illustrate this, in \cref{fig:colliderstaged} the upward edge emanating from $v_2$ can be labeled ${P(Y=0\mid X=0)}$ and the upward edge from $v_4$ as ${P(Z=0\mid X=0,Y=0)}$.
Here, the vertices $v_2$ and $v_3$ are in the same stage. This models the independence $X\independent Y$ via the equalities ${P(Y=y\mid X=0) = P(Y=y\mid X=1)}$ for $y=0,1$. Similarly in \cref{fig:staged}, the 
additional stages together
entail $Z\independent X\mid Y$.


Thus, staged tree models include as a special case both discrete Bayesian networks  and context-specific discrete Bayesian networks, modelling conditional independence relations that hold only for a subset of the states of a discrete random variable \citep{Boutilier.etal.1996}. In particular, all conditional independence relations given by the Bayesian network can be reflected in the staging as described above. \citet{Varando2021} give an algorithm that for any Bayesian network $G$ outputs a staged tree $\T_G$ representing the same model.


Because probability trees and staged trees may have root-to-leaf paths of different lengths and because there are no constraints on where stages can be imposed, the class of staged tree models is much wider than the one of discrete Bayesian networks.
Working within this class unlocks an immediate graphical representation of modeling assumptions such as the space of events, the intricate conditional independence relationships within this space, and the local sum-to-one conditions.  
Because staged trees grow quickly even for small problems, whenever they exhibit many symmetries in their subtrees they are sometimes represented more compactly by alternative graphical depictions called \emph{chain event graphs}: see the textbook by \citet{Collazo.etal.2018} for an in-depth discussion of these points.

\section{Staged trees are curved exponential families}
\label{sec:main}

\subsection{Exponential families}
\label{sub:expfam}

Following \citet{Kass.Vos.1997}, a parametric statistical model 
$\{\ptheta\mid\theta\in\Theta\}$ 
on a space $\mathcal{X}$ is called an \emph{exponential family} if every distribution in the model can be written in the form
\begin{equation}\label{eq:expfam}
\ptheta(x)=h(x)\exp\Bigl(\eta(\theta)^\top T(x)-\psi(\theta)\Bigr) \qquad\text{for all }x\in\mathcal{X}
\end{equation}
where $\eta:\Theta\to\R^d$ is the canonical parameter, $T:\mathcal{X}\to\R^d$ is a minimally sufficient statistic, $^\top$ denotes the transpose operation, and ${\psi:\Theta\to\R}$ is the cumulant-generating function. The space $\mathcal{N}=\{{\eta(\theta)\in\R^d}\mid\text{\cref{eq:expfam} is integrable}\}$ is called the \emph{natural parameter space}. If $\mathcal{N}$ is an open and non-empty subset of $\R^d$ then the model is called a \emph{regular} exponential family of dimension $d$.
  
For exponential families, moments of all orders exist and the maximum likelihood estimator exists and is unique. Regular exponential families are closed under linear contraints on the natural parameter space and the log-likelihood function on that space is concave. Under more general constraints on the parameters, these families are more technical to study but can often retain many useful asymptotic properties. The two best-studied generalizations of regular exponential families are so-called curved and, more generally, stratified exponential families. The development of this paper requires only the former type of models which we formally introduce in the third subsection below.

\subsection{Regular exponential families}
\label{sub:regular}

In a first step, we focus on probability trees without imposing a stage structure. In this context, the probability tree whose graph is a root vertex connected to its leaves exclusively via single edges is a \emph{star} as in the graph-theoretic sense.

\begin{lemma}\label{lem:saturated}
Every probability tree model on $n$ atoms is equal to the full probability simplex $\Delta_{n-1}^\circ$.

In particular, the star with $n$ leaves and with probability $\theta_{1r}$ attached to edge $r$, for $r=1,\ldots,n$, represents the Multinomial distribution $\Multi(1,\theta)$ with one trial and parameters $\theta=(\theta_{11},\ldots,\theta_{1n})$.
\end{lemma}

\begin{proof}
Every probability tree with $n$ root-to-leaf paths specifies a set of distributions inside the open $n-1$-dimensional probability simplex as outlined in \cref{sub:probtrees}. Conversely, given a fixed tree graph with $n$ leaves, every point $(p_1,\ldots,p_n)$ inside the simplex can be interpreted as a vector whose $r$th component $p_r$ is the probability of going down the $r$th root-to-leaf path in the tree, $r=1,\ldots,n$. We can pick the label of the $j$th edge out of vertex $v_i$ to be the fraction $\nicefrac{\sum_{b\in[ij]}p_b}{\sum_{a\in[i]}p_a}$ where $[i],[ij]\subseteq\{1,\ldots,n\}$ denote the indices of all root-to-leaf paths passing through vertex $v_i$ or through the tail of its $j$th outgoing edge, respectively. These labels are conditional probabilities and their product along a root-to-leaf path is equal precisely to the atomic probability of that path. This proves the first claim.

As for the second claim, $\ptheta(x)=\prod_{r=1}^n\theta_{1r}^{\alpha_{1r}(x)}$ where $\sum_{r=1}^n\theta_{1r}=1$ is the probability distribution induced by the star with edges numbered $r=1,\ldots,n$. This is the stated Multinomial distribution.
\end{proof}

\citet{Wishart.1949}\label{wish} derives the exponential form \cref{eq:expfam} for the Multinomial distribution $\Multi(1,\theta)$ as follows: $h\equiv1$ is constant, the natural parameters are normalized log-probabilities $\eta_r(\theta)=\log(\nicefrac{\theta_{1r}}{(1-\sum_{s=1}^{n-1}\theta_{1s})})$ for $r=1,\ldots,n$, the sufficient statistic is the vector of the first $n-1$ edge indicators $T=(\alpha_{11},\ldots,\alpha_{1,n-1})$, and the cumulant-generating function is the logarithm of the normalizing constant $\psi(\theta)={\log(1-\sum_{s=1}^{n-1}\theta_{1s})}$.

By \cref{lem:saturated}, all probability trees on the same number of root-to-leaf paths are statistically equivalent in the sense that they all represent the same model, namely the full probability simplex. Thus, \emph{any} probability tree is a graphical representation of the Multinomial distribution, and Wishart's result gives a sufficient statistic, cumulant-generating function, and the natural parameters for any probability tree, possibly after reparametrization.

In Proposition~1, we give a parametrization of the exponential family of a probability tree model which is alternative to Wishart's. This new parametrization has the advantage of respecting the structure of trees having root-to-leaf paths longer than single edges.
%
%
Retaining this extra structure
allows us to generalize the exponential-family parametrization of probability trees to the case of staged trees in Section~\ref{sub:curved}.

We thus derive the following result. Here for any probability tree $\T$ we say that the $\kappa_i$th edge emanating from vertex $v_i$ points \emph{downwards}, $i=1,\ldots,k$. We then recursively define functions $N_i:\Theta_{\T}\to\R$ as $N_i(\theta)=1$ for leaf vertices $v_i$ and else as a product of labels pointing downwards $N_i(\theta)={(1-\sum_{s=1}^{\kappa_i-1}\theta_{is})N_{i\kappa_i}(\theta)}$ for $i=1,\ldots,k$. The shorthand $N_{ij}$ denotes $N_r$ for the vertex $v_r$ which is the tail of the $j$th edge coming out of vertex $v_i$. \Cref{fig:natural-parameters} illustrates this notation.

\begin{figure}
\centering
\begin{tikzpicture}
\renewcommand{\xx}{2.8}
\renewcommand{\yy}{0.7}
\newcommand{\vertx}[1]{\tikz{\node[shape=circle,draw,inner sep=1.5pt]{$v_{#1}$};}} 
\newcommand{\vertxx}[1]{\tikz{\node[shape=circle,draw,inner sep=0.5pt]{$v_{#1}$};}} 
\newcommand{\vertxxx}[1]{\tikz{\node[shape=circle,draw,inner sep=-0.5pt]{$v_{#1}$};}} 
\newcommand{\somevertx}{\tikz{\node[shape=circle,draw,inner sep=2.5pt]{};}} 

\node (v1) at (0.2*\xx,0*\yy) {\vertx{1}};
\node (vi) at (1*\xx,0*\yy) {\vertx{i}};
\node (vij) at (2*\xx,1*\yy) {\vertxx{ij}};
\node (vijd1) at (3*\xx,0*\yy) {\somevertx};
\node (vijd2) at (3.5*\xx,-0.5*\yy) {\somevertx};
\node (vijd3) at (4*\xx,-1*\yy) {\somevertx};
\node (vik) at (2*\xx,-1*\yy) {\vertxxx{i\kappa_i}};
\node (vikd1) at (3*\xx,-2*\yy) {\somevertx};
\node (vikd2) at (3.5*\xx,-2.5*\yy) {\somevertx};
\node (vikd3) at (4*\xx,-3*\yy) {\somevertx};

\draw[->,dashed] (v1) -- (vi);

\draw[->,very thick] (vi) -- node [pos=0.6,fill=white] {$\theta_{ij}$} (vij);
\draw[->,dashed] (vi) -- (1.8*\xx,1.5*\yy);
\draw[->,dashed] (vi) -- (1.9*\xx,0.3*\yy);
\draw[->,very thick] (vi) -- node [below] {$1-\sum_{s=1}^{\kappa_i-1}\theta_{is}$\hspace*{1cm}} node [above,pos=0.5] {$\boldsymbol{N_i}$} (vik);

\draw[->,very thick] (vij) -- node [above,pos=0.3] {$\boldsymbol{N_{ij}}$} (vijd1);
\draw[->,dashed,very thick] (vijd1) -- (vijd2);
\draw[->,very thick] (vijd2) -- (vijd3);

\draw[->,very thick] (vik) --node [below] {$1-\sum_{t=1}^{\kappa_{\cdot}-1}\theta_{\cdot t}$\hspace*{1cm}} (vikd1);
\draw[->,dashed,very thick] (vikd1) -- (vikd2);
\draw[->,very thick] (vikd2) --node [pos=0.3,below=4pt] {\ldots} (vikd3);
\end{tikzpicture}
\caption{Illustration of the notation used in \cref{prop:probtrees}. The natural parameter $\eta_{ij}$ belonging to the indicator $\alpha_{ij}$ of \enquote{passing through the edge $(v_i,v_{ij})$} is a function of the label $\theta_{ij}$ of that edge and of the products $N_i$ and $N_{ij}$ of the labels of the downwards pointing $v_i$- and $v_{ij}$-to-leaf paths, respectively.
The edges involved in the definition of $\eta_{ij}$ are depicted in bold.
\label{fig:natural-parameters}}
\end{figure}
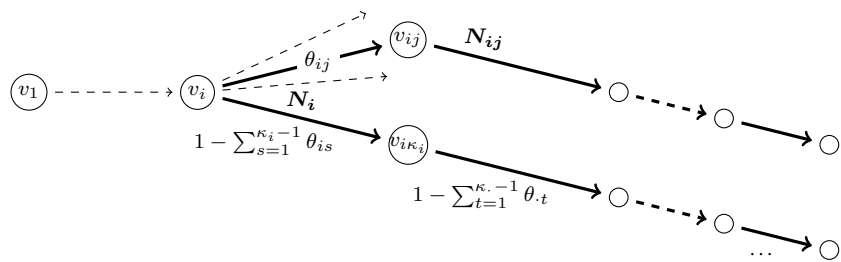

\begin{proposition}\label{prop:probtrees}
Every probability tree $\mathcal T$ with parameters $\theta$ represents a regular exponential family with the following attributes:
\begin{itemize}
\item the indicators $T_{ij}=\alpha_{ij}$ of the first $j=1,\ldots,\kappa_i-1$ edges of all inner vertices $i=1,\ldots,k$ are a sufficient statistic,
\item the natural parameters $\eta_{ij}$ are locally normalized log-probabilities defined by $\eta_{ij}(\theta)=\log(\theta_{ij}\nicefrac{N_{ij}(\theta)}{N_i(\theta)})$ for all $j=1,\ldots,\kappa_i-1$ and $i=1,\ldots,k$, and the natural parameter space is $\R^{d}$ with $d=\sum_{i=1}^k(\kappa_i-1)$, and
\item the cumulant-generating function is the negative log-sum of normalizing constants along the root-to-leaf path whose edges all point downwards, $\psi(\theta)=-\log(N_1(\theta))$.
\end{itemize}
\end{proposition}

\begin{proof}
The desired parametrization is equivalent to
\begin{equation}\label{eq:p-eta}
\ptheta(x)=N_1(\theta)\prod_{i=1}^k\prod_{j=1}^{\kappa_i-1}\left(\theta_{ij}\frac{N_{ij}(\theta)}{N_i(\theta)}\right)^{\alpha_{ij}(x)}\quad \text{for all $x$}.
\end{equation}

To prove this equality, we first rewrite the probability mass function introduced in \cref{sec:ST}. For any inner vertex $i=1,\ldots,k$ in the probability tree, the label of the $\kappa_i$th outgoing edge is a function of the first ${1,\ldots,\kappa_i-1}$ edges, namely $\theta_{i\kappa_i}=1-\sum_{j=1}^{\kappa_i-1}\theta_{ij}$. The indicator $\alpha_{i\kappa_i}$ of passing along that edge is a function of the indicators of those same edges and equals $\alpha_i(x)(1-\sum_{j=1}^{\kappa_i-1}\alpha_{ij}(x))$ where $\alpha_i(x)$ is one if $x$ reaches vertex $v_i$ and zero otherwise.

Thus, the probability mass function of a probability tree can be written as
\begin{equation}
\ptheta(x)=\prod_{i=1}^k\prod_{j=1}^{\kappa_i}\theta_{ij}^{\alpha_{ij}(x)}
=\prod_{i=1}^k\prod_{j=1}^{\kappa_i-1}\theta_{ij}^{\alpha_{ij}(x)}\biggl(1-\sum_{j=1}^{\kappa_i-1}\theta_{ij}\biggr)^{\alpha_i(x)(1-\sum_{j=1}^{\kappa_i-1}\alpha_{ij}(x))}.
\end{equation}
And as a consequence, the claim reduces to:
\begin{equation}\label{eq:claim}
\prod_{i=1}^k\biggl(1-\sum_{j=1}^{\kappa_i-1}\theta_{ij}\biggr)^{\alpha_i(x)(1-\sum_{j=1}^{\kappa_i-1}\alpha_{ij}(x))}
=N_1(\theta)\prod_{i=1}^k\prod_{j=1}^{\kappa_i-1}\left(\frac{N_{ij}(\theta)}{N_i(\theta)}\right)^{\alpha_{ij}(x)}
\end{equation}
for all root-to-leaf paths $x$. We now prove this claim by induction on the number of inner vertices $k$.

Let thus in a first step the graph be a star, $k=1$. If $x$ is the root-to-leaf path which is one of the first $1,\ldots,\kappa_1-1$ edges coming out of the root then \cref{eq:claim} states that $1=N_1(\theta)\cdot\nicefrac{1}{N_1(\theta)}=1$. If otherwise $x$ equals the $\kappa_1$st edge then $1-\sum_{j=1}^{\kappa_1-1}\theta_{1j}=N_1(\theta)$ which is also true.

If $k>1$ then $x$ may have length greater one and we distinguish two analogous cases. Hereby the induction hypotheses holds for trees smaller than the one we consider and so, in particular, the claim is true for its subtrees. Without loss, we number the vertex at the tail of the first edge of $x$ as $v_2$. For simplicity, $A$ denotes the left hand side of \cref{eq:claim} and $B$ the right hand side of that equation. Then:

Case $1$: The edge $(v_1,v_2)$ is one of the first $1,\ldots,\kappa_1-1$ edges coming out of the root. Then
\begin{equation}
\begin{split}
A&=1\cdot\prod_{i=2}^k\biggl(1-\sum_{j=1}^{\kappa_i-1}\theta_{ij}\biggr)^{\alpha_i(x)(1-\sum_{j=1}^{\kappa_i-1}\alpha_{ij}(x))}\\
&=N_2(\theta)\prod_{i=2}^k\prod_{j=1}^{\kappa_i-1}\left(\frac{N_{ij}(\theta)}{N_i(\theta)}\right)^{\alpha_{ij}(x)}\\
&=N_1(\theta)\frac{N_2(\theta)}{N_1(\theta)}\prod_{i=2}^k\prod_{j=1}^{\kappa_i-1}\left(\frac{N_{ij}(\theta)}{N_i(\theta)}\right)^{\alpha_{ij}(x)}=B
\end{split}
\end{equation}
where the final step is true because $N_2(\theta)=N_{12}(\theta)$.

Case $2$: The edge $(v_1,v_2)$ is the $\kappa_1$st edge coming out of the root. Then
\begin{equation}
\begin{split}
A&=\biggl(1-\sum_{j=1}^{\kappa_1-1}\theta_{1j}\biggr)\prod_{i=2}^k\biggl(1-\sum_{j=1}^{\kappa_i-1}\theta_{ij}\biggr)^{\alpha_i(x)(1-\sum_{j=1}^{\kappa_i-1}\alpha_{ij}(x))}\\
&=\biggl(1-\sum_{j=1}^{\kappa_1-1}\theta_{1j}\biggr)N_2(\theta)\prod_{i=2}^k\prod_{j=1}^{\kappa_i-1}\left(\frac{N_{ij}(\theta)}{N_i(\theta)}\right)^{\alpha_{ij}(x)}\\
&=N_1(\theta)\prod_{i=2}^k\prod_{j=1}^{\kappa_i-1}\left(\frac{N_{ij}(\theta)}{N_i(\theta)}\right)^{\alpha_{ij}(x)}=B
\end{split}
\end{equation}
where the final step is true because $\prod_{j=1}^{\kappa_1-1}\left(\nicefrac{N_{1j}(\theta)}{N_1(\theta)}\right)^{\alpha_{1j}(x)}=1$. 

This proves \cref{eq:p-eta}, and every probability tree represents an exponential family in this parametrization.

The regularity claim follows since the map $\eta$ is a diffeomorphism between the space of model parameters $\Theta$ and $\R^d$. Indeed, it has a smooth inverse obtained by first computing $p_{\eta}$ for a given $\eta\in\R^d$ according to the given parametrization and then computing $\theta\in\Theta$ such that $p_{\eta}=\ptheta$ as described in the proof of \cref{lem:saturated}.
\end{proof}

\Cref{prop:probtrees} enables us to simply read the parametrisation of the underlying exponential family directly from a given probability tree. 
%

\begin{example}\label{exp:exp}
Consider \cref{fig:staged&CEG} and the probability tree which has the same graph as the staged trees in \cref{fig:colliderstaged} and~\ref{fig:staged}. For simplicity, in this binary tree we do not use double indices but label the upwards edges going out of vertex $v_i$ as $\theta_i$ and the downward edges as $1-\theta_i$ for $i=1,\ldots,7$. 

The indicator functions $\alpha_{i1}$ of the upwards edges $i=1,\ldots,7$ are a sufficient statistic for this tree. The corresponding natural parameters are then log-ratios of the downwards labels derived as
\begin{align}
\eta_1(\theta)&=\log(\theta_1\nicefrac{(1-\theta_2)(1-\theta_5)}{(1-\theta_1)(1-\theta_3)(1-\theta_7)}),\\\eta_2(\theta)&=\log(\theta_2\nicefrac{(1-\theta_4)}{(1-\theta_2)(1-\theta_5)}),\\ \eta_3(\theta)&=\log(\theta_3\nicefrac{(1-\theta_6)}{(1-\theta_3)(1-\theta_7)}),
\end{align}
and $\eta_i(\theta)=\log(\nicefrac{\theta_i}{(1-\theta_i}))$ for $i=4,\ldots,7$. The cumulant-generating function is the logarithm of the product of the labels along the downward root-to-leaf path coming out of the root vertex $\psi(\theta)=-\log\left[(1-\theta_1)(1-\theta_3)(1-\theta_7)\right]$.
\end{example}

\subsection{Curved exponential families}
\label{sub:curved}

As stated in \cref{sub:stagedtrees}, every discrete Bayesian network has a corresponding staged tree representation. Since for instance the collider graph does not represent a regular exponential family \citep{Koller.Friedman.2009}, staged trees cannot in general be regular exponential families either. They rather form what is called a \emph{curved} exponential family: a submodel of a regular exponential family whose parameter space is a smooth manifold \citep{Efron.1978}.

\begin{theorem}\label{thm:CEGcurved}
Staged tree models are curved exponential families.
\end{theorem}

\begin{proof}
By \cref{prop:probtrees}, every staged tree model is a submodel of a regular exponential family. We prove that the natural parameter space is always a smooth manifold of the right dimension by showing that it is the image of a certain linear subspace of $\Theta$ under the diffeomorphism $\eta.$

Let $\T_0$ denote a probability tree with parameter space $\Theta_0=\times_{i=1}^{k}\Delta_{\kappa_i-1}^\circ$. Let $\T$ denote the same tree graph together with an imposed stage structure and parameter space $\Theta_{\T}=\times_{r\in R}\Delta_{\kappa_r-1}^\circ$ for some index set $R\subseteq\{1,\ldots,k\}$.
The parameter space of the staged tree model is the  kernel of the parameter space of the saturated model under the linear function $h_{\T}:\R^{d_0}\to\R^{d_0-d}$ which encodes the ${d_0-d}$ identifications of edge labels as
\begin{equation}\label{eq:linearh}
h_{\T}(\theta)=\bigl(\theta_{ij}-\theta_{st}~\mid~\text{for all }j=t=1,\ldots,\kappa_j\text{ and all }v_i \text{ and } v_s \text{ in the same stage}\bigr)
\end{equation}
where $d=\sum_{r\in R}(\kappa_r-1)$ is the number of free parameters in the staged tree, and $d_0=\sum_{i=1}^k(\kappa_i-1)$ is the number of free parameters in the probability tree. By construction, $h_{\T}(\theta)=0$ if and only if $\theta$ fulfills the stage constraints in $\T$. That is, the parameter space of the staged tree model equals the kernel of the map \cref{eq:linearh}, so $h_{\T}^{-1}(0)=\Theta_{\T}$.

Since $h_{\T}$ is surjective, its kernel is $d$-dimensional. Thus, the parameter space of the staged tree model is a $d$-dimensional linear space in $\R^{d_0}$.
As a consequence, the natural parameter space, obtained as the image of the diffeomorphism $\eta$ given in \cref{prop:probtrees}, is a smooth manifold of the same dimension. The claim follows.
\end{proof}

As an alternative proof strategy for \cref{thm:CEGcurved} one may use the characterization of staged tree models as the solution set of a collection of polynomial equations and inequalities provided by \citet{Duarte.Goergen.2019}. Whenever these polynomials do not exhibit any algebraic singularities inside the probability simplex, the model is a curved exponential family: compare the discussion of implicit model representations given in \citet{Geiger.etal.2001} and of algebraic exponential families in \citet{Drton.Sullivant.2009}.

In particular, whilst the equations coding stage constraints in the conditional probabilities as in \cref{eq:linearh} are always linear, those coding the same constraints formulated in terms of the natural parameters in general are not.
In the following proposition we describe these constraints and explain when they are linear.


\begin{proposition}\label{prop:stages} Let $\T$ be a staged tree. For every vertex $v_i$ and outgoing edge indexed by $j$, let $\eta_{ij}$ denote the natural parameter of the exponential form derived in \cref{prop:probtrees}, where we additionally define $\eta_{i\kappa_i} = 0$. Let further $\xi_{ij} = \exp(\eta_{ij})$ and
$P_{ij} = \sum_{r}\prod_{ab}\xi_{ab}^{\alpha_{ab}(r)}$ where $r$ ranges over the $v_{ij}$-to-leaf paths for every fixed double index $ij$.

Then two vertices $v_i$ and $v_s$ are in the same stage if and only if they have the same number $\kappa$ of emanating edges and for all $j = 1,\dotsc,\kappa$ the equality
\begin{equation}\label{eq:stages}
\eta_{ij} + \log(P_{ij}) + \log(P_{s\kappa}) = \eta_{sj} + \log (P_{sj}) + \log (P_{i\kappa})
\end{equation}
is true.

Furthermore, let $N_{ij}$ denote the product of the labels along the downwards-pointing $v_{ij}$-to-leaf path as in \cref{prop:probtrees}. Then \cref{eq:stages} is a linear expression in the $\eta$ if and only if $P_{ij}P_{s\kappa}=P_{sj}P_{i\kappa}$, or equivalently
\begin{equation}\label{eq:regularcondition}
N_{ij}N_{s\kappa} = N_{sj}N_{i\kappa}.
\end{equation}
\end{proposition}

\begin{proof}
By definition, $\xi_{ij} = \theta_{ij}N_{ij}/N_i$. For any $i$ and edge $j$ emanating from $v_i$, the equalities
\begin{equation}
N_{ij}P_{ij} = \sum_{r}N_{ij}\prod_{ab}\xi_{ab}^{\alpha_{ab}(r)} = \sum_{r}\prod_{ab}\theta_{ab}^{\alpha_{ab}(r)} = 1
\end{equation}
are true, where $r$ ranges over the $v_{ij}$-to-root paths. In particular, the second equality is obtained by applying \cref{eq:p-eta} to the subtree rooted at $v_{ij}$. The third equality holds because its left-hand side is the sum of all atomic probabilities associated to the subtree rooted at $v_{ij}$. Now let $v_i$ and $v_s$ be in the same stage. Equivalently, $\theta_{ij} = \theta_{sj}$ for all $j=1,\dotsc,\kappa$. This is equivalent to $\xi_{ij}N_i/N_{ij} = \xi_{sj}N_s/N_{sj}$ which in turn is equivalent to $\xi_{ij}N_{i\kappa}/N_{ij} = \xi_{sj}N_{s\kappa}/N_{sj}$ simply because $\theta_{i\kappa} = \theta_{s\kappa}$. Using $N_{ij}P_{ij} = 1$, rearranging, and taking logarithms yields the statement in~\cref{eq:stages}.

As for the second claim, if \cref{eq:regularcondition} holds then the log terms in \cref{eq:stages} vanish, making it a linear expression in the $\eta$. Conversely, let \cref{eq:stages} be a linear expression in the $\eta$. Then the expression
\begin{equation}\label{eq:notlinear}
\log(P_{ij}P_{s\kappa}/P_{sj}P_{i\kappa})
\end{equation}
is also linear in the $\eta$. This implies
$
P_{ij}P_{s\kappa}/P_{sj}P_{i\kappa} = c\prod_{ab}\xi^{\beta_{ab}}
$
for some ${c, \beta_{ab} \in \mathbb R}$. Since the left-hand side of this equality is a rational function in the $\xi$, all $\beta_{ab}$ must be integers. Indeed, otherwise there would exist a choice of $\xi$ that would make the right-hand side evaluate to an irrational number whilst the left-hand side can only evaluate to rational numbers. So, we obtain an equation of the form
\begin{equation}
P_{ij}P_{s\kappa}\prod_{ab\in A} \xi_{ab}^{\beta_{ab}} = P_{sj}P_{i\kappa}\prod_{cd\in B} \xi_{cd}^{\beta_{cd}}
\end{equation}
for some sets $A, B$ of indices and $\beta_{ab},\beta_{cd}\in \mathbb N$. 

Suppose now that one of $A$ and $B$ is nonempty, without loss of generality assume it to be $B$. Then some $\xi_{cd}$ divides $P_{ij}$. By the definition of $P_{ij}$, the only way this is possible is if $\xi_{cd}$ divides all summands of $P_{ij}$. That is, if all paths of the subtree rooted at $v_{ij}$ (relabeled with $\theta\mapsto\xi$) have the label $\xi_{cd}$ in common. But this is impossible because in a staged tree we assume that each node has at least two children, and thus we can always find a path that avoids the label $\xi_{cd}$.
Hence we see that both $A$ and $B$ must be empty, so $P_{ij}P_{s\kappa}=P_{sj}P_{i\kappa}$. And hence $N_{ij}N_{s\kappa} = N_{sj}N_{i\kappa}$ using $N_{ij}P_{ij} = 1$.
\end{proof}

Let $\mathcal T$ be a staged tree. \cref{prop:stages} shows that the exponential family of $\mathcal T$ proposed in this paper is a regular exponential family if and only if the equality of ratios in \cref{eq:regularcondition} is true for every pair of vertices $v_i$ and $v_s$ in the same stage and every edge $j$ emanating from these vertices. For this reason, staged trees satisfying this property are of particular interest. We henceforth call these \emph{regular staged trees}.

Equation \cref{eq:regularcondition} gives a simple criterion for regularity by comparing certain concatenations of downwards-pointing paths. Algebraically, it can be checked by comparing the monomials $\prod_{ab} \theta_{ab}$ obtained by multiplying all labels in the four downward-pointing paths starting from $v_{ij}, v_{i\kappa}, v_{sj}, v_{s\kappa},$ respectively. Graphically, it can be verified by checking that the concatenation of the first and fourth path on the above list is the same as the concatenation of the second and third, up to a permutation of the edges.

Next, we explore two prominent classes of regular staged trees. The notion of a balanced staged tree, first formalized in \cite{Duarte.Ananiadi.2020}, is an important notion for studying the toric geometry of staged trees \citep{Goergen.etal.2021}. It makes use of the interpolating polynomials $t_{ij}$ defined by $t_{ij} = \sum_{r}\prod_{ab}\theta_{ab}^{\alpha_{ab}(r)},$ where $r$ ranges over the $v_{ij}$-to-leaf paths. This sum is to be read as a formal sum of labels, i.e.\ without using local sum-to-one conditions. (Including these conditions would imply $t_{ij} = 1$.)
A staged tree is \emph{balanced} if for all $v_{i}$ and $v_s$ in the same stage and all edges $j$ emanating from $v_i$ we have $t_{ij}t_{s\kappa}=t_{i\kappa}t_{sj}$, where again $\kappa$ denotes the chosen downwards-pointing edge emanating from $v_i$. This definition for balanced staged trees differs slightly from the one found in the literature but is nevertheless equivalent to it. For a graphical interpretation of this condition, note that $t_{ij}t_{s\kappa}$ is the interpolating polynomial of the staged tree obtained by attaching a copy of the subtree rooted at $v_{s\kappa}$ to each of the leaves of the subtree rooted at $v_{ij}$. The condition for a balanced tree now is equivalent to this composite tree being statistically equivalent to the composite tree corresponding to the product $t_{i\kappa}t_{sj}$.

Simple staged trees are defined in \cite{Collazo.etal.2018}. A staged tree is \emph{simple} if for all $v_{i}$ and $v_s$ in the same stage and all edges $j$ emanating from $v_i$ we have $t_{ij}=t_{sj}$.

\begin{proposition}\label{prop:balancedregular}
All balanced staged trees are regular. In particular, all simple staged trees are regular.
\end{proposition}
\begin{proof}
By definition, all simple trees are balanced. Now let $\mathcal T$ be a balanced staged tree and $v_i$, $v_s$ vertices in the same stage. Let $j$ be an edge emanating from $v_{i}$. Splitting off the unique downwards-pointing path from $v_{ij}$, write $t_{ij} = t'_{ij} + N_{ij}$, and likewise for the other pairs of indices.
Then
\begin{align}
0
&= t_{ij}t_{s\kappa}-t_{i\kappa}t_{sj}
\\ & = (t'_{ij}t'_{s\kappa} - t'_{sj}t'_{i\kappa} + t'_{ij}N_{s\kappa} - t'_{sj}N_{i\kappa} + t'_{s\kappa}N_{ij} - t'_{i\kappa}N_{sj}) + (N_{ij}N_{s\kappa} - N_{sj}N_{i\kappa})
\\ & =: R + Q
\end{align}
where we define $R$ resp.\ $Q$ to be the left resp.\ right outer summand of the middle expression.
Consider this expression as a polynomial $P$ in the $\theta_{ab}$ labels. Split this set of labels into two sets $A$ and $B$ of labels pointing downwards, resp.\ not pointing downwards. Thus $P$ can be viewed as a polynomial in $\mathbb R[A][B]$. Now, all summands of $P$ in $R$ are divisible by some label in the set $B$. This means that $Q\in\mathbb R[A]$ is the constant term of the polynomial $P\in \mathbb R[A][B]$. Thus $P = 0$ implies $Q = 0$.
\end{proof}

The following converse to \cref{prop:balancedregular} holds for all binary trees with three levels.

\begin{conjecture}
All regular staged trees are balanced.
\end{conjecture}

As an indication for this conjecture, while \cref{eq:regularcondition} seems much weaker than the balanced condition for two nodes $v_i$ and $v_s$, in a regular tree 
it can be applied recursively everywhere downstream of $v_i$ and $v_s$, greatly limiting the possible staging structure.

\begin{example}\label{exp:exp2}
\Cref{fig:colliderstaged} shows a staged tree representation of the collider Bayesian network ${X\to Z\leftarrow Y}$ which is not a regular exponential family. We can see here that indeed in our parametrization, the linear stage identifications $\theta_2=\theta_3$ do not give rise to linear constraints on the natural parameters $\eta_2$ and $\eta_3$. Instead, by \cref{prop:stages},
\begin{equation}
\eta_2+\log(\xi_4+1)+\log(\xi_7+1)=\eta_3+\log(\xi_5+1)+\log(\xi_6+1)
\end{equation}
where $\xi_l=\exp(\eta_l)$ for $l=1,\ldots,7$.
This implies the following equation
\begin{equation}
\exp(\eta_2)(1+\exp(\eta_4))(1+\exp(\eta_7))=\exp(\eta_3)(1+\exp(\eta_5))(1+\exp(\eta_6))
\end{equation}
which fully characterises the corresponding curved exponential family in the natural parameters.

The staged tree in \cref{fig:staged} however fulfills the linearity criteria in \cref{prop:stages}. In particular, the stage constraints $\theta_2=\theta_3$, $\theta_4=\theta_6$, and $\theta_5=\theta_7$ together imply that $(1-\theta_4)(1-\theta_7)=(1-\theta_6)(1-\theta_5)$ as in \cref{eq:regularcondition}. They thus give rise to linear (equality) constraints $\eta_2=\eta_3$ on the natural parameters. 
The same simplification would occur were $v_4$ and $v_5$, and $v_6$ and $v_7$ in the same stage, respectively, rather than $v_4$ and $v_6$, and $v_5$ and~$v_7$. These stagings all give rise to balanced and regular trees.
\end{example}

Equations on the natural parameters in a curved exponential family can be highly non-trivial. In this paper we found that for a staged tree they are functions of subgraphs, derived from the conditional independence relations coloured in the tree. Thus, the equations on its exponential family parametrization can be directly read from the graph.
This formulation has only been possible thanks to the probability tree's expressiveness of the underlying space of events and of its parametrization. 

Analogous results have to the best of the authors' knowledge not been derived in the literature of Bayesian networks. While sufficient statistics for these models are known, natural parameters have not been explicitly computed \citep[e.g.][]{Loh.Wainwright.2013}. 
In order to translate our formulae into the language of Bayesian networks, first a given directed acyclic graph needs to be transformed into the corresponding staged tree using the algorithm provided by \citet{Varando2021}, and then the staged-tree language can be used to infer a natural parametrization as in \cref{prop:probtrees}. Bayesian networks themselves provide a too-compact representation of the underlying modelling assumptions to be able to directly express their exponential form as a function of the graph. However, \cref{prop:balancedregular} yields the following insight.

\begin{corollary}
Let $G$ be a Bayesian network. The exponential family obtained by applying Proposition 1 to the associated staged tree $\T_G$ is regular if G is decomposable or if there exists a topological ordering $(1,\dotsc,n_{G})$ of the nodes of $G$ such that  the parents of the $(i+1)$th node are a subset of the union of the $i$th node and its parents for all $i$.
\end{corollary}

\begin{proof}

The first statement is well known but also a direct consequence of the fact that $\T_G$ is balanced if and only if $G$ is decomposable \citep{Duarte.Solus.2021}. The second statement corresponds to the definition of $G$ being \emph{simple} given in \citet{Leonelli.Varando.2021}. These authors prove that $G$ is simple if and only if $\T_G$ is simple. By \cref{prop:balancedregular}, simple staged trees form regular exponential families.
\end{proof}


\section*{Acknowledgements}
We are grateful to Giovanni Pistone and to Piotr Zwiernik for discussions during the early stages of this project and to Eva Riccomagno for comments on an earlier version of this paper. We also thank the two anonymous referees for their comments which led to a significant improvent of our results.

Christiane G\"orgen and Manuele Leonelli were supported by the programme \enquote{Oberwolfach Leibniz Fellows} of the Mathematisches Forschungsinstitut Oberwolfach in 2017. Orlando Marigliano was supported by International Max Planck Research School and Brummer \& Partners MathDataLab.


\bibliographystyle{imsart-nameyear} 
\bibliography{ourbibitems}       


\end{document}